\documentclass[twoside, 10pt]{article}
\usepackage{fullpage}
\usepackage{pgf}
\usepackage{graphicx}
\usepackage{amsmath}
\usepackage{amsfonts}
\usepackage{amssymb}
\usepackage{amsxtra}
\usepackage{amstext}
\usepackage{latexsym}
\usepackage{dsfont}
\usepackage{stmaryrd}
\usepackage{mathrsfs}

\newtheorem{theorem}{Theorem}

\newtheorem{corollary}[theorem]{Corollary}

\newtheorem{lemma}[theorem]{Lemma}

\newtheorem{proposition}[theorem]{Proposition}

\newenvironment{proof}[1][\noindent Proof]{\textbf{#1} }{\hfill \ \rule{0.5em}{0.5em}} \newenvironment{remark}[1][\noindent
Remark]{\noindent\textbf{#1} }

\def\C{\mathcal{C}}

\def\F{\mathcal{F}}

\def\H{\mathcal{H}}

\def\K{\mathcal{K}}

\def\Q{\mathcal{Q}}

\def\U{\mathcal{U}}

\def\n#1{\|{#1}\|}
\def\bin#1,#2{#1\choose#2}

\def\nn{\mathbb{N}}

\def\rr{\mathbb{R}}

\def\oo{\emptyset}

\title{\bf Condensing and Semi-continuous Multi-functions on Uniform Spaces}

\author{Ra\'ul Fierro$^{1,2}$\thanks{Corresponding author: Ra\'ul Fierro, Instituto de Matem\'atica, Pontificia Universidad Cat\'olica de
Valpara\'{\i}so, Casilla 4059, Valpara\'{\i}so, Chile. Email: rfierro@ucv.cl; rafipra@gmail.com}\\[1ex]
{\footnotesize $^{1}$Instituto de Matem\'atica. Pontificia Universidad Cat\'{o}lica de Valpara\'{\i}so.}\\[-0.5ex] {\footnotesize $^{2}$Instituto de Matem\'aticas. Universidad de Valpara\'{\i}so.}}

\begin{document} \maketitle

\noindent{\bf Abstract. }  Some concepts, such as non-compactness measure and condensing operators, defined on metric spaces are extended to uniform spaces. Such extensions allow us to locate, in the context of uniform spaces, some classical results existing in nonlinear analysis. An application of our results is given for operators defined on locally convex spaces. The main aim of this work is to unify some well-known results existing in complete metric and vector topological spaces.

\noindent{\bf Keywords}: uniform space, semi-continuity, Hausdorff pseudometric, non-compactness measure, condensing operator, fixed point, inward condition.

\bigskip

\noindent\textbf{MSC2010:}{ Primary 54E15, 47H04,47H10; Secondary 06A06\hfill}
%\noindent {\sl 2000 Mathematics Subject Classification.} Primary: %60H10, Secondary: 34F05, 93E03 \pagebreak

%%% ----------------------------------------------------------------------

%%% ----------------------------------------------------------------------
\section{Introduction}
The concept of non-compactness measure was introduced by Kuratowski in \cite{Ku30} to characterize relative compact sets as those whose non-compactness measure equals zero. This measure was used by Sadowskii in \cite{Sa67} to define condensing function in a Banach space, whereby he generalized the well-known Schauder fixed point theorem \cite{Sc30}. Later, by extending the definition of non-compactness measure to locally convex spaces, in \cite{HiEtAl69} Himmelberg \emph{et al.} proved a fixed point theorem for condensing functions in locally convex spaces, generalizing even the corresponding result by Tychonov in \cite{Ty35}. Nowdays,  condensing functions and multi-functions have been defined on metric and topological vector spaces and significant results have been obtained.
On the other hand, since the Banach contraction principle \cite{Ba22} and Brouwer work \cite{Br12} were published, fixed point theory has been developed for operators defined on complete metric and topological vector spaces.  Some remarkable works on this subject, among others, are the articles of Darbo \cite{Da55}, Kakutani \cite{Ka41}, Schauder \cite{Sc30} and Tychonoff \cite{Ty35}. Even though completeness is defined for both metric and topological linear spaces, some metric spaces are not linear and some linear spaces are not metrizable. Due to the convenience of establishing results as the aforementioned in a unified way, the main aim of this paper is addressing attention in developing concepts such as non-compactness measure and condensing multi-function into a structure containing both of these class of spaces, namely the uniform structure. Uniform spaces have been studied by various authors in different contexts. The importance of these structures lies in the fact that a wide variety of classes of topological spaces are included in this category. For example, metric spaces, Hausdorff topological vector spaces and Menger (probabilistic) metric spaces are particular cases of uniform spaces. In \cite{Fa96} Fang proved that these classes are $F$-type topological spaces, which, according to Hamel \cite{HL03} (see also \cite{Ha05}), coincides with the category of uniform spaces.

We are interested in studying topological properties of multi-functions taking values in a uniform space. To this end, the Hausdorff topology on the closed and bounded subsets of the space is considered. It is well-known that this topology is generated by a family of pseudo metrics (Hausdorff pseudometrics), consequently it is uniform as well. We use these concepts to define weak lower and upper semi-continuity for multi-functions, moreover their relation with the continuity with respect to the Hausdorff topology is studied. By extending the definition of non-compactness measure to uniform spaces, we are able to characterize the relative compactness of subsets as those sets having non-compactness measure equal to zero and condensing functions and multi functions can be defined. Such extensions allow us to locate, in the context of uniform spaces, some classical results existing in nonlinear analysis. Caristi and Bishop and Phelps results have been extended, by some authors, for operators taking values in a uniform space (see for example, \cite{Br74,Ha05,HL03,Mi90,MT89}). We revisit some of these works in order to expand the range of situations where new versions of these results can be applied.

Including this introduction, the paper is divided in six sections.  In
Section 2, some preliminary definitions, notations and facts are stated in order to prove results in subsequence
sections. Section 3 is devoted to study relations among the different type of semi-continuities and the continuity of multi-functions according to the Hausdorff topology. The non-compactness and condensing concepts for uniform spaces are introducing in Section 4, along with some of the main results of this paper. In Section 5, extended versions of Caristi's theorem are presented and finally, we devote Section 6 for an application of some our results to multi-functions defined on locally convex spaces.

\section{Preliminaries}
In all of this work, $(X,\U)$ stands for a uniform space. For each $x\in X$, $A\subseteq X$ and $U\in \U$, we denote $U[x]=\{y\in X:(x,y)\in U\}$ and $U[A]=\bigcup_{x\in A}U[x]$. We consider $X$ endowed with the topology induced by $\U$.
A subset $B$  of $X$ is said to be bounded, if there exist $x\in X$ and $U\in \mathcal{U}$ such that $B\subseteq U[x]$. We denote by $\mathcal{B}(X)$ the family of all bounded subsets of $X$ and by $\C(X)$ the family of all closed and nonempty subsets of $X$. Let $\mathcal{CB}(X)=\mathcal{C}(X)\cap \mathcal{B}(X)$. On $\mathcal{CB}(X)$ a family $\H=\{H_U\}_{U\in\U}$ is defined as follows:
$$
H_U=\{(A,B)\in\mathcal{CB}(X)\times\mathcal{CB}(X):A\subseteq U[B]\mbox{ and }B\subseteq U[A]\}.
$$
It is well known, see \cite{CV77} for instance, that  $\H$ is a fundamental system of entourages for a uniformity on $\mathcal{CB}(X)$. The topology on $\mathcal{CB}(X)$ induced by $\H$ is referred to as the $\H$-topology  and continuity with respect  to this topology is referred to as $\H$-continuity. Let $\K(X)$ denote the family of all nonempty compact subsets of $X$. Since $\K(X)\subseteq \mathcal{CB}(X)$, we consider $\K(X)$ endowed with the induced $\H$-topology. In \cite{SR69/70} it is proved that $\mathcal{K}(X)$ is $\H$-complete, if and only if, $(X,\U)$ so is.

It is well-known that the uniformity $\U$ is generated by a family of pseudo metrics $\{d_\lambda\}_{\lambda\in\Lambda}$ on $X$. As usual, for each $\lambda\in\Lambda$, $a\in A$, $\epsilon>0$ and $A\in\mathcal{CB}(X)$, we denote $\mathrm{B}_\lambda(a,\epsilon)=\{x\in X:d_\lambda(a,x)<\epsilon\}$ and $\mathrm{B}_\lambda(A,\epsilon)=\bigcup_{x\in A}\mathrm{B}_\lambda(x,\epsilon)$.
It is assumed that the family $\{d_\lambda\}_{\lambda\in\Lambda}$ is saturated, which means that the following condition holds:
\begin{description}
  \item[(S)] For each $\lambda,\mu\in\Lambda$, there exists $\nu\in\Lambda$ such that $\max\{d_\lambda(x,y),d_\mu(x,y)\}\leq d_\nu(x,y)$, for all $x,y\in X$.
\end{description}
This condition implies that a set $G\subseteq X$ is open, whenever for each $a\in G$ there exist $\lambda\in\Lambda$ and $\epsilon>0$ such that $\mathrm{B}_\lambda(a,\epsilon)\subseteq G$. As showed in Section \S1.2, Chapter IX in \cite{Bo66b}, condition (S) is not a restriction on the topology of $X$. That is, when the family of pseudo metrics does not satisfy condition (S), we can find such a family satisfying this condition and generating the same topology.

It is easy to see (see Theorem II-12 in \cite{CV77}, for instance) that the $\H$-topology on $\mathcal{CB}(X)$ is generated by the family of pseudo metrics $\{H^\lambda\}_{\lambda\in\Lambda}$  defined by
$$
H^\lambda(A,B)=\max\left\{\sup_{x\in A}d_\lambda(x,B),\sup_{y\in B}d_\lambda(y,A)\right\},
$$
where, as usual for $A\in\mathcal{CB}(X)$, $a\in A$ and $\lambda\in\Lambda$,  $d_\lambda(a,A)=\inf_{x\in A}d_\lambda(a,x)$.  In the sequel, these notations are maintained.

\begin{remark}
Observe that, for all $A,B\in\mathcal{CB}(X)$ and $\lambda\in\Lambda$,
$$
H^\lambda(A,B)=\inf\{\epsilon>0:A\subseteq \mathrm{B}_\lambda(B,\epsilon),B\subseteq \mathrm{B}_\lambda(A,\epsilon)\}.
$$
\end{remark}

Let $E$ be a set and $T:E\to \C(X)$  be a multi-function. For any $B\subset X$, the inverse  image of $B$ under $T$ is defined by $T^{-1}(B)=\{x\in E:Tx\cap B\neq\oo\}$ and for each $A\subseteq E$ the image of $A$ by $T$ is defined as $T(A)=\bigcup_{x\in A}Tx$. Next, suppose $E$ is a
topological space. The multi-function  $T$ is said to be lower (respectively,  upper) semi-continuous, if for any open subset $G$ (respectively, closed subset
$F$) of $X$, $T^{-1}(G)$ is an open (respectively, $T^{-1}(F)$ is a closed) subset of  $E$.

\section{Semi-continuity and $\H$-continuity}

\begin{proposition}\label{p1} Let $E$ be a topological space and $T:E\to\mathcal{CB}(X)$ an $\mathcal{H}$-continuous multi-function.
Then, $T$ is lower semi-continuous.
\end{proposition}
\begin{proof}
Let $a\in E$ and $G$ be an open subset of $X$ such that $Ta\cap G\neq\oo$. Hence,
there exist $x_0\in Ta$ and $U\in\U$ such that $Ta\cap U^{-1}[x_0]\neq\oo$ and $U^{-1}[x_0]\subseteq
G$. Since $T$ is continuous at $a$, with respect to the $\H$-topology, there exists $V_a$ neighborhood of $a$ such that $(Tu,Ta)\in H_U$, if $u\in V_a$. Hence for each $u\in V_a$, we have $x_0\in Ta\subseteq U[Tu]$ and consequently $\oo\neq Tu\cap U^{-1}[x_0]\subseteq Tu\cap G$. This concludes the proof.
\end{proof}

Let $\K(X)$ denote the family of all compact nonempty subsets of $X$. Since $\K(X)\subseteq \mathcal{CB}(X)$, we consider $\K(X)$ endowed with the
induced $\H$-topology.

\begin{theorem}\label{t2}
Let $E$ be a topological space and $T:E\to\K(X)$ a multi-function.  Then, $T$ is lower  and upper semi-continuous, if and only if, $T$
is continuous with respect to the induced $\H$-topology on $\K(X)$.
\end{theorem}
\begin{proof}
Let us assume $T$ is a lower and upper semi-continuous
multi-function. Let $a\in E$, $U\in\U$ and $V\in\U$ such that $V\circ V^{-1}\subseteq U$. We have
$Ta\subseteq V[Ta]$ and due to $T$ is upper semi-continuous, there exists $V_a$
neighborhood of $a$ such that, for each $u\in V_a$, $Tu\subseteq V[Ta]\subseteq (V\circ V^{-1})[Ta]$. On the other hand, $Ta$ is compact and hence there exist $x_1,\dots,x_r\in Ta$ such that $Ta\subseteq G$, where $G=V[x_1]\cup\dots\cup V[x_r]$. Since for all $i\in\{1,\dots,r\}$, $Ta\cap V[x_i]\neq\oo$, the lower semi-continuity of $T$ implies there exists
$W_a$ neighborhood of $a$ such that, for each $u\in W_a$, $Tu\cap V[x_i]\neq\oo$, for all $i\in\{1,\dots,r\}$. That is, for
each $u\in W_a$, there exist $y_1,\dots,y_r\in Tu$ such that for all $i\in\{1,\dots,r\}$, $y_i\in V[x_i]$. Hence
$$
Ta\subseteq\bigcup_{i=1}^rV[x_i]\subseteq\bigcup_{i=1}^r(V\circ V^{-1})[y_i]\subseteq\bigcup_{y\in Tu}U[y].
$$
We have proved that for each $u\in V_a\cap W_a$, $Tu\subseteq U[Ta]$ and $Ta\subseteq U[Tu]$. Thus, $T$ is continuous at $a$ according with the $\H$-topology.

Next, we assume $T$ is continuous according to the $\H$-topology on $\K(X)$. From  Proposition \ref{p1}, $T$ is lower semi-continuous. Let us prove that
$T$ is upper semi-continuous. Let $a\in X$ and $G$ be an open subset of $X$ such that $Ta\subseteq G$.  Since $X$ is  a uniform space and $Ta$ is
compact (see Proposition 4, Section \S4.3, Chapter II in \cite{Bo66}, for instance), there exists $U\in\U$ such that $U[Ta]\subseteq G$. By assumption, there exists $V_a$ neighborhood of $a$ such that $Tu\subseteq U[Ta]$ (and $Ta\subseteq U[Tu]$) if $u\in
V_a$. Hence, for all $u\in V_a$, $ T(u)\subseteq G$, i.e., $T$ is upper semi-continuous and therefore the proof is
complete.
\end{proof}

Let $D$ be a subset of $X$ and $T: D\to\C(X)$ a multi-function. We say $T$ is weakly lower (respectively, upper) semi-continuous, if  for each $\lambda\in\Lambda$ and $\alpha\geq0$, the set $\{x\in D:d_\lambda(x,T(x))<\alpha\}$ (respectively, $\{x\in D:d_\lambda(x,T(x))>\alpha\}$) is open, and we say $T$ is weakly continuous, if  $T$ is both weakly lower and weakly upper semi-continuous.

\begin{theorem}\label{t3}
Let $D$ be a subset of $X$ and $T: D\to\C(X)$ a  multi-function.
\begin{description}
\item[(\ref{t3}.1)]If $T$ is
lower semi-continuous, then $T$ is weakly lower semi-continuous.
\item[(\ref{t3}.2)] If $T$ is upper semi-continuous, then $T$ is weakly upper
semi-continuous.
\end{description}
\end{theorem}
\begin{proof}
Let $\lambda\in \Lambda$ and $\alpha \geq0$. Suppose $T$ is lower semi-continuous and
define $A=\{x\in D:d_\lambda(x,Tx)<\alpha \}$. In order to prove that $A$ is an open set, suppose $A\neq \emptyset $ and choose $a\in A$. Let
$\epsilon =\alpha -d_\lambda(a,Ta)$ and $y\in Ta$ such that $d_\lambda(a,y)<\epsilon /3+d_\lambda(a,Ta)$. Since $T$ is lower semi-continuous,
there exists a neighborhood  $V'_a$ of $a$ such that $Tu\cap B_\lambda(y,\epsilon /3)\neq \emptyset $, for all  $u\in V'_a$. Let
$V_a=V'_a\cap B_\lambda(a,\epsilon/3)$, $u\in V_a$ and $ b_{u}\in Tu\cap B_\lambda(y,\epsilon /3)$. We have $d_\lambda(u,b_{u})\leq
d_\lambda(u,a)+d_\lambda(a,y)+d_\lambda(y,b_{u})<\alpha $ and consequently, $d_\lambda(u,Tu)<\alpha$. This proves that $A$ is an open set and
therefore, (\ref{t3}.1) holds.

Next we define $A=\{x\in D:d_\lambda(x,Tx)>\alpha\}$. Let $a\in A$ and
choose $\beta, \gamma\in\mathbb{R}$ such that $\gamma>\beta>\alpha$ and $d_\lambda(a,Ta)>\gamma$. Let $G=\{x\in
D:d_\lambda(x,Ta)<(\gamma-\beta)/2\}$. Since $Ta\subseteq G$, $G$ is open and $T$ is upper semi-continuous, there exists $V'_a$ neighborhood of
$a$ such that for each $x\in V'_a$, $Tx\subseteq G$. This implies that for each $x\in V'_a$ and each $y\in Tx$,
$d_\lambda(y,Ta)<(\gamma-\beta)/2$. Hence, $$ \gamma<d_\lambda(a,Ta)\leq d_\lambda(a,y)+d_\lambda(y,Ta)<d_\lambda(a,y)+(\gamma-\beta)/2. $$
Thus, $d_\lambda(a,y)>(\gamma+\beta)/2$ and consequently, $d_\lambda(a,Tx)\geq(\gamma+\beta)/2$. Let $V_a=V'_a\cap B_\lambda(a,\beta-\alpha)$ and
note that for each $x\in V_a$, $$ \beta<d_\lambda(a,Tx)\leq d_\lambda(a,x)+d_\lambda(x,Tx)<\beta-\alpha+d_\lambda(x,Tx). $$ This proves that
$V_a\subseteq A$ and therefore, $A$ is an open set, which concludes the proof.
\end{proof}

\begin{remark}
As shown in \cite{FMO11}, even though $X$ is a metric space, there exist weakly continuous  multi-functions, which are not
lower or upper semi-continuous multi-functions.
\end{remark}
\begin{theorem}\label{t4} Let $E$ be a subset of $X$ and $T:E\to\mathcal{CB}(X)$ be a $\H$-continuous multi-function. Then, $T$ is
weakly continuous.
\end{theorem}
\begin{proof}
Let $\lambda\in\Lambda$, $u,v\in E$, $\epsilon>0$ and $y_\epsilon\in T(v)$ such that
$d_\lambda(v,y_\epsilon)<d_\lambda(v,T(v))+\epsilon$. Hence,
$$
\begin{array}{ccl} d_\lambda(u,T(u)) & \leq & d_\lambda(u,v)+d_\lambda(v,y_\epsilon)+
d_\lambda(y_\epsilon,T(u))\\ &< & d_\lambda(u,v)+d_\lambda(v,T(v))+\epsilon+H_\lambda(T(v),T(u)).\\
\end{array}
$$
Consequently,
$$
|d_\lambda(u,T(u))-d_\lambda(v,T(v))|\leq d_\lambda(u,v)+H_\lambda(T(v),T(u))
$$
and therefore, the weak continuity of $T$ is directly obtained.
\end{proof}

\section{Non-compactness measure}
In accordance with the uniformity generating the topology of $X$, a subset $D$ of $X$ is precompact, if for each  $\lambda\in\Lambda$ and $\epsilon>0$, there exist $x_1,\dots,x_r\in X$ such that $D\subseteq \mathrm{B}_\lambda(x_1,\epsilon)\cup\dots \cup\mathrm{B}_\lambda(x_r,\epsilon)$. Moreover, a filter $\F$ on $X$ is  a Cauchy filter (c.f. \cite{Bo66}), if for each $\lambda\in\Lambda$ and $\epsilon>0$, there exist $A\in\F$  such that $A\times A\subseteq \mathrm{U}(\lambda,\epsilon)$, where $\mathrm{U}(\lambda,\epsilon)=\{(x,y)\in X\times X:d_\lambda(x,y)<\epsilon\}$.

A non-compactness measure $\alpha:\mathcal{B}(X)\to[0,\infty[$ is defined as
$\alpha(A)=\inf\{\epsilon>0:\C(\epsilon)\ni A\}$, where for $\epsilon>0$, $\C(\epsilon)$ is the family of all $A\in\mathcal{B}(X)$ such that  for
each $\lambda\in\Lambda$, there exist $x_1,\dots, x_r\in X$ such that $A\subseteq \mathrm{B}_\lambda(x_i,\epsilon/2)\cup\cdots\cup \mathrm{B}_\lambda(x_r,\epsilon/2)$. From Theorem 3, Section \S 4.2, Chapter II in \cite{Bo66}, $A\in \mathcal{B}(X)$ is precompact, if and only if, $\alpha(A)=0$.

It is easy to see the following three properties hold:
\begin{description}
\item[a)]$\alpha(A)\leq\alpha(B)$ whenever $A\subseteq B$,\quad($B\in
\mathcal{B}(X)$).
\item[b)] $\alpha(A\cup B)=\max\{\alpha(A),\alpha(B)\}$,\quad ($A,B\in \mathcal{B}(X)$).
\item[c)] $\alpha(\overline{A})=\alpha(A)$,\quad ($A\in \mathcal{B}(X)$).
\item[d)] $\alpha\left(\bigcap_{\lambda\in\Lambda}\mathrm{\mathrm{B}}_\lambda(A,\epsilon)\right)\leq\alpha(A)+\epsilon$,\quad ($A\in \mathcal{B}(X),\epsilon>0$).
\end{description}

Theorem \ref{t5} below extends a classical result (Theorem 1') by Kuratowski in \cite{Ku30}.

\begin{theorem}\label{t5}
Suppose $X$ is complete. Let $\mathcal{B}$ be a filter base on $X$ such that $\mathcal{B}\subseteq \mathcal{CB}(X)$, $E=\bigcap_{B\in\mathcal{B}}B$,
and suppose $\inf\{\alpha(B):B\in\mathcal{B}\}=0$. Then, $E$ is compact and nonempty, and $\mathcal{B}$ converges to $E$ in the $\H$-topology, i.e.\ for each $\lambda\in\Lambda$  and $\epsilon>0$, there exists $B\in\mathcal{B}$ such that $H^\lambda(B,E)<\epsilon$.
\end{theorem}
\begin{proof}
Since $E$ is closed, $X$ is complete and $\alpha(E)=0$, we have $E$ is compact.  Let $\F$ be  the filter generated by
$\mathcal{B}$, i.e., $\F=\{A\subseteq X:\exists B\in\mathcal{B}, B\subseteq A\}$. In order to prove $E$ is nonempty, let us denote by $\F^*$ a
ultrafilter such that $\F\subseteq\F^*$. By assumption, for each $\epsilon>0$ there exists $B\in\mathcal{B}$ such that for each $\lambda\in\Lambda$, there exist $x_1,\dots,x_{r}\in X$ satisfying
$B\subseteq\mathrm{B}_\lambda(x_1,\epsilon/2)\cup\cdots\cup\mathrm{B}_\lambda(x_r,\epsilon/2)$.
Since $B\in\F^*$,  one of these balls $\mathrm{B}_\lambda(x_1,\epsilon/2),\dots,\mathrm{B}_\lambda(x_r,\epsilon/2)$ belongs to $\F^*$ (see Corollary in Section \S 6.4, Chapter I in
\cite{Bo66}, for instance). Thus, $\F^*$ is a Cauchy filter and consequently it converges to some point $x^*\in X$. But,
$\{x^*\}=\bigcap_{F\in\F^*}F\subseteq \bigcap_{F\in\F}F= \bigcap_{B\in\mathcal{B}}B$ and therefore $E$ is non empty.

Next, we prove the convergence of $\mathcal{B}$ to $E$ in the $\H$-topology. Suppose there exist  $\lambda\in\Lambda$ and $\epsilon>0$ such that for any $B\in\mathcal{B}$, $H^\lambda(B,E)>\epsilon$. Let $L(\epsilon,\lambda)=\{x\in X:d_\lambda(x,E)\geq\epsilon\}$. Since $H^\lambda(B,E)=\sup_{x\in
B}d_\lambda(x,E)$, there exists $x\in B$ such that $d_\lambda(x,E)>\epsilon$. Thus, for any $B\in\mathcal{B}$, $B\cap
L(\epsilon,\lambda)\ne\emptyset$. This fact implies, $\widetilde{\mathcal{B}}=\{B\cap L(\epsilon,\lambda):B\in\mathcal{B}\}$ is a filter base on the complete uniform space
$L(\epsilon,\lambda)$. Moreover $\inf\{\alpha(B):B\in\widetilde{\mathcal{B}}\}=0$. Hence, by applying to  $\widetilde{\mathcal{B}}$ what we have shown at the first paragraph to $\mathcal{B}$,  we have
$\oo\ne\bigcap_{B\in\widetilde{\mathcal{B}}}B=\bigcap_{B\in\mathcal{B}}B\cap L(\epsilon,\lambda)=E\cap L(\epsilon,\lambda)$, which is a contradiction. Therefore, the proof is complete.
\end{proof}

\begin{remark} Note that whether $\mathcal{B}$ is countable, in Theorem \ref{t5}, in order to
$\bigcap_{B\in\mathcal{B}}B$ is compact and nonempty,  it suffices that $X$ is sequentially complete.
\end{remark}

Let $T:D\subseteq X\to 2^X$ be a multi-function. A point $x\in D$ is said to be  a fixed point of $T$, whenever $x\in Tx$. We say that $T$ is
condensing, if $T$ is continuous and for each $A\subseteq D$ such that $A\in \mathcal{B}(X)$ and $\alpha(A)>0$, we have $T(A)\in \mathcal{B}(X)$ and $\alpha(T(A))<\alpha(A)$.

\begin{proposition}\label{p6}
Suppose $X$ is complete. Let $D$ be a closed subset of $X$, $T:D\subseteq X\to \mathcal{CB}(X)$ be a weakly  upper semi-continuous
multi-function and for each $\lambda\in\Lambda$ and $\eta>0$,  $B_{\lambda,\eta}=\{x\in D:d_\lambda(x,Tx)\leq\eta\}$. Suppose for each $\lambda\in\Lambda$ and $\eta>0$,  $B_{\lambda,\eta}\neq\oo$ and $\inf\{\alpha(B_{\lambda,\eta}):\lambda\in\Lambda,\eta>0\}=0$. Then, $T$ has a fixed point.
\end{proposition}
\begin{proof}
Since $\{d_\lambda\}_{\lambda\in\Lambda}$ satisfies condition (S) and $T$ is a weakly upper semi-continuous multi-function,
$\mathcal{B}=\{B_{\lambda,\eta}:\lambda\in\Lambda,\eta>0\}$ is a filter base of closed and  bounded subsets of the complete space $D$. Hence, Theorem \ref{t5} implies $\bigcap_{\lambda\in\Lambda,\eta>0}B_{\lambda,\eta}\neq\oo$ and therefore, due to $T$ has closed images, $T$ has a fixed point.
\end{proof}

Given a filter base $\mathcal{A}$ on a subset $D$ of $X$, we denote by $\mathcal{A}'$ the  set of all its accumulation points, i.e.\
$x\in\mathcal{A}'$, if for any neighborhood $V$ of $x$ and $A\in\mathcal{A}$, $V\cap A\ne\emptyset$. An extension of an old result due to Furi Vignoli
\cite{FV69} can be stated as follows.

\begin{proposition}\label{p7}
Let $D$ be a complete subset of $X$, $T:D\subseteq X\to \mathcal{CB}(X)$ a condensing  weakly upper semi-continuous multi-function and $\mathcal{A}\subseteq\mathcal{B}(X)$ a filter base on $D$. Suppose $\inf\{\sup_{x\in A,\lambda\in\Lambda}d_\lambda(x,Tx):A\in
\mathcal{A}\}=0$. Then, $\mathcal{A}'\neq\emptyset$ and each point in $\mathcal{A}'$ is a fixed point of $T$.
\end{proposition}
\begin{proof}
Let $\epsilon>0$. From assumption, there exists $A\in\mathcal{A}$  such that $\sup_{x\in A,\lambda\in\Lambda}d_\lambda(x,Tx)<\epsilon$. This fact
implies  $A\subseteq \bigcap_{\lambda\in\Lambda}B_\lambda(T(A),\epsilon)$ and hence, property d) implies that $ \alpha(A)\leq\alpha(T(A))+\epsilon. $ Thus,
$\alpha(A)\leq\alpha(T(A))$ and since $T$ is condensing, we have $\alpha(A)=0$. The completeness of $D$ implies that $A$ is compact and consequently $\mathcal{A}'=\bigcap_{A\in\mathcal{A}}\overline{A}\neq\emptyset$. Note that from the assumption, for each  $\eta>0$, there exists
$A_{\eta}\in\mathcal{A}$ such that $A_{\eta}\subseteq \bigcap_{\lambda}B_{\lambda,\eta}$, where $B_{\lambda,\eta}$ is defined as in Proposition \ref{p6}. Since $T$ is weakly upper semi-continuous, $\overline{A}_{\eta}\subseteq \bigcap_{\lambda}B_{\lambda,\eta}$ and therefore, $\mathcal{A}'\subseteq \bigcap_{\lambda\in\Lambda,\eta>0}B_{\lambda,\eta}$, concluding the proof due to $T$ has closed images.
\end{proof}

We state below an extension of a known result in metric spaces.

\begin{theorem}\label{t8}
Let $D\in \mathcal{CB}(X)$ and $T:D\to 2^D$ be a condensing multi-function. Then, there exists a compact subset $C$ of $D$ such that $T(C)\subseteq C$.
\end{theorem}
\begin{proof}
Let $x_0\in D$ and $\Sigma=\{K\in\mathcal{CB}(X):x_0\in K\subseteq D\mbox{ and } T(K)\subseteq K\}$. Due to $D\in\Sigma$, we have $\Sigma\neq\oo$. Let $B=\bigcap_{K\in\Sigma}K$ and $C=\overline{\{x_0\}\cup T(B)}$. We have $T(B)\subseteq\bigcap_{K\in\Sigma}T(K)\subseteq B$ and $x_0\in B$. Moreover, since $B$ is closed, $C\subseteq B$. Thus $T(C)\subseteq T(B)\subseteq  C$, $C\in\Sigma$ and $B=C$. From properties b) and c), we obtain $\alpha(C)=\alpha(T(B))=\alpha(T(C))$  and due to $T$ is condensing, we have $\alpha(C)=0$ and the proof is complete.
\end{proof}

Let $\F=\{k_\lambda\}_{\lambda\in\Lambda}$ be a family of constants such that for each $\lambda\in\Lambda$, $0\leq
k_\lambda<1$, $D$ a subset of $X$ and $T:D\to \mathcal{CB}(X)$. We say $T$ is an $\F$-contractive multi-function, if for any $x,y\in D$ and $\lambda\in\Lambda$, $H^\lambda(Tx,Ty)\leq k_\lambda d_\lambda(x,y)$. Let $k$ be a constant such that $0<k<1$.  We say $T$ is a $k$-set contraction, if $T$ is continuous and for each $A\subseteq D$ such that $A\in \mathcal{B}(X)$, we have $T(A)\in \mathcal{B}(X)$ and $\alpha(T(A))\leq k\alpha(A)$. Of course, every $k$-set contraction is condensing.

\begin{theorem}\label{t9}
Let $\F=\{k_\lambda\}_{\lambda\in\Lambda}$ be a family of nonnegative constants such that  $k=\sup_{\lambda\in\Lambda}k_\lambda<1$ and $T:X\to \mathcal{K}(X)$ be an $\F$-contractive multi-function. Then, $T$ is a $k$-set contraction.
\end{theorem}
\begin{proof}
Let $A\in\mathcal{CB}(X)$ and $\epsilon=\alpha(A)$. For each $\lambda\in\Lambda$ and $\eta>0$, there exist $a_1,\dots,a_r\in X$ such that $A\subseteq A^\lambda_1\cup\cdots\cup A^\lambda_r$, where $A^\lambda_i=\mathrm{B}_\lambda(a_i,\eta+\epsilon/2)$, for $i\in\{1,\dots,r\}$.   Hence, for each $x\in A^\lambda_i$, we have
$H^\lambda(Tx,Ta_i)\leq k_\lambda d_\lambda(x,a_i)<k_\lambda(\eta+\epsilon/2)< \eta+k\epsilon/2$. Thus, $Tx\subseteq \mathrm{B}_\lambda(Ta_i,\eta+k\epsilon/2)$ and $T(A^\lambda_i)\subseteq\mathrm{B}_\lambda(Ta_i,\eta+k\epsilon/2)$. Since $Ta_i$ is compact, there exist $b_{i1},\dots,b_{is}\in Ta_i$ such that $Ta_i\subseteq \mathrm{B}_\lambda(b_{i1},\eta)\cup\cdots\cup\mathrm{B}_\lambda(b_{is},\eta)$. Hence, $T(A^\lambda_i)\subseteq\mathrm{B}_\lambda(b_{i1},2\eta+k\epsilon/2)\cup\cdots\cup\mathrm{B}_\lambda(b_{is},2\eta+k\epsilon/2)$ and consequently, $\alpha(T(A^\lambda_i))\leq 2\eta +k\epsilon$.  But $T(A)\subseteq T(A^\lambda_1)\cup\cdots\cup T(A^\lambda_r)$ along with property b) imply $\alpha(T(A))\leq 2\eta+k\epsilon$, for all $\eta>0$. Accordingly, $\alpha(T(A))\leq k\alpha(A)$ and  the proof is complete.
\end{proof}

\begin{remark}
Let $B\in \mathcal{CB}(X)$ be a non compact set and $T:X\to\mathcal{CB}(X)$ be defined as $Tx=B$, for all $x\in X$. We have, for any family $\F$ of nonnegative constants, $T$ is an $\F$-contractive multi-function, which is not condensing. This example shows that it is not a strong condition, in Theorem \ref{t9}, to assume that $T$ has compact images.
\end{remark}

A well-known result by Nadler \cite{Na69} is generalized as follows.

\begin{theorem}\label{t15}
Suppose $X$ is sequentially complete and let $T:X\to \mathcal{CB}(X)$ be an  $\F$-contractive multi-function with $\F=\{k_\lambda;\lambda\in\Lambda\}$ a family of nonnegative constants such that
$\sup_{\lambda\in\Lambda}k_\lambda<1$. Then, $T$ has a fixed point.
\end{theorem}
\begin{proof} Let
$\{x_n\}_{n\in\nn}$ be a Cauchy sequence with respect to the metric $\rho$ defined, for $x,y\in X$, as
$\rho(x,y)=\sup_{\lambda\in\Lambda}d_\lambda(x,y)\wedge1$. Hence, $\{x_n\}_{n\in\nn}$ is a Cauchy sequence with respect to the uniformity $\U= \{\mathrm{U}(\lambda,\epsilon);\lambda\in\Lambda, \epsilon>0\}$ and since $X$ is
sequentially complete, there exists $x\in X$ such that $\{x_n\}_{n\in\nn}$ converges to $x$. Let
$\epsilon>0$ and $N\in\nn$ such that $\rho(x_m,x_n)<\epsilon$ whether
$m,n\geq N$. From the lower semi-continuity of $\rho(x_m,\cdot)$, we
have $\rho(x_m,x)\leq\epsilon$, for each $m\geq N$. This proves that $\{x_n\}_{n\in\nn}$ converges to $x$ according to $\rho$ and
consequently $(X,\rho)$ is a complete metric space.

Let $H^\rho$ be the Hausdorff metric on $\mathcal{CB}(X)$ defined from $\rho$. It is easy to see that for each $A,B\in\mathcal{CB}(X)$,
$H^\rho(A,B)=\sup_{\lambda\in\Lambda}H_\lambda(A,B)\wedge1$ and consequently, according to $H^\rho$, $T$ is a contraction. It follows from Theorem 5
by Nadler \cite{Na69} that there exists $x^*\in X$ such that  $x^*\in T(x^*)$, which completes the proof.
\end{proof}

\section{Extended versions of Caristi's theorem}

In this section, the space $X$ is assumed complete with respect to $\U$, i.e.\ any Cauchy filter base in $X$ converges.

\begin{lemma}\label{le10}
 Let $\{\varphi_\lambda\}_{\lambda\in\Lambda}$ be a family of lower semi-continuous and  bounded below functions from $X$ to
$\rr$ and $\preceq$ be a relation on $X$ defined as follows: $u\preceq v$, if and only if, for all $\lambda\in\Lambda$, $d_\lambda(u,v)\leq
\varphi_\lambda(u)-\varphi_\lambda(v)$. Then, $\preceq$ is a partial order relation on $X$ and for each $x_0\in X$, there exists a maximal element
$x^*\in X$ such that $x_0\preceq x^*$.
\end{lemma}
\begin{proof} It is easy to see that $\preceq$ is a partial order relation on $X$. For each $x\in
X$, let $I(x)=\{y\in X:x\preceq y\}$. Since for each $\lambda\in\Lambda$,  $\varphi_\lambda$ is lower semi-continuous and
$I(x)=\bigcap_{\lambda\in\Lambda}\{y\in X: \varphi_\lambda(y)+d_\lambda(x,y)\leq \varphi_\lambda(x)\}$, we have $I(x)$ is closed. Fix $x_0\in X$, let $C$ be a
totally ordered subset of $I(x_0)$ and $\mathcal{B}=\{I(x)\cap C:x\in C\}$. Hence, $\mathcal{B}$ is a filter base in $I(x_0)$. Let $\epsilon>0$.
Because of for each $\lambda\in\Lambda$, $\varphi_\lambda$ is bounded below, there exists $L_\lambda=\inf\{\varphi_\lambda(x):x\in C\}$. Let us choose
$x_\lambda\in C$ such that $\varphi_\lambda(x_\lambda)<L_\lambda+\epsilon$ and notice $\varphi_\lambda$ is decreasing. Consequently, for
$x_{\lambda}\preceq u\preceq v$, we have $\varphi_\lambda(u)-\varphi_\lambda(v)<\epsilon$ and hence $d_\lambda(u,v)<\epsilon$. This fact proves
$\mathcal{B}$ is a Cauchy filter base in $I(x_0)$ and thus it converges to some $v\in I(x_0)$. Since for each $x\in C$, $I(x)$ is closed, we have
$\{v\}=\bigcap_{x\in C}\overline{I(x)\cap C}\subseteq \bigcap_{x\in C}I(x)$ and hence $v$ is an upper bound of $C$. Therefore, by Zorn's Lemma there
exists a maximal element $x^*\in I(x_0)$, concluding de proof.
\end{proof}

\begin{remark} Lemma of Mizoguchi in \cite{Mi90} could have been used to prove Lemma \ref{le10}. However,  for the sake of completeness, we preferred to give an independent proof. On the other hand, when for each $\lambda\in\Lambda$, $\varphi_\lambda=\varphi$ does not depend on $\lambda$, the proof of the above lemma could be carried out by defining a suitable sequence instead of a filter base. In this case, this lemma requires only sequentially completeness instead of completeness.
\end{remark}

Theorem \ref{th11} below gives two extended versions of Caristi's theorem. In particular, Theorem 1 by Mizogushi in \cite{Mi90}, which is an improved version of Caristi's Theorem  \cite{Ca76}, is generalized by means of (\ref{th11}.1) below.

\begin{theorem}\label{th11}  Let $T:X\to \C(X)$ be a multivalued function  and suppose for each $\lambda\in\Lambda$,  $\varphi_\lambda:X\to\rr$ is a lower
semi-continuous and bounded below function. Then,  the following two propositions hold:
\begin{description}
\item[(\ref{th11}.1)] If for  each $x\in X$, there exists $y\in Tx$ such that for  each $\lambda\in\Lambda$, $d_\lambda(x,y)\leq \varphi_\lambda(x)-\varphi_\lambda(y)$, then, $T$ has a fixed point.
\item[(\ref{th11}.2)] If for  each $x\in X$, each $y\in Tx$ and each $\lambda\in\Lambda$,  $d_\lambda(x,y)\leq
\varphi_\lambda(x)-\varphi_\lambda(y)$, then, there exists $x^*\in X$ such that $\{x^*\}= Tx^*$.
\end{description}
\end{theorem}
\begin{proof}
From Lemma \ref{le10}, there exists a maximal element $x^*\in X$ and from assumption in (\ref{th11}.1), there exists $y\in Tx^*$  such that $x^*\preceq y$. Since $x^*$ is maximal, we have $x^*=y$ and hence $x^*\in Tx^*$. Therefore,  (\ref{th11}.1) holds.

Let us prove (\ref{th11}.2). Let $x^*\in X$ the maximal element of $X$. Since $Tx^*$ is nonempty, assumption in (\ref{th11}.2) implies that there exists $y\in Tx^*$ such that $x^*\preceq y$. Thus $x^*=y$ and $\{x^*\}\subseteq Tx^*$. By applying assumption in (\ref{th11}.2) again and the maximality of $x^*$, we have $T(x^*)\subseteq \{x^*\}$ and the proof is complete.
\end{proof}

When the functions $\varphi_\lambda$, in Theorem \ref{th11}, does not depend on $\lambda\in\Lambda$,  this theorem coincides with Theorem 2 by Hamel in \cite{Ha05}. Even, in this case, it is only required  $X$ to be sequentially complete. However, as we think in the proof of Corollary \ref{c13} below, it is a great limitation to assume all the functions $\varphi_\lambda$ are equal.

The following corollary follows from Theorem \ref{th11} and  is an extension of the  well-known theorem by Caristi \cite{Ca67} for single valued functions.
\begin{corollary}\label{c12} Let $f:X\to X$ be an arbitrary function. Suppose  for each $\lambda\in\Lambda$, $\varphi_\lambda:X\to\rr$ is a  lower
semi-continuous and bounded below function such that for each $x\in X$, $d_\lambda(x,f(x))\leq \varphi_\lambda(x)-\varphi_\lambda(f(x))$. Then, there
exists $x^*\in X$ such that $f(x^*)=x^*$.
\end{corollary}

Corollaries \ref{c13} and \ref{c14} below are extensions of  the classical Banach contraction principle.

\begin{corollary}\label{c13}
 Let $\F=\{k_\lambda\}_{\lambda\in\Lambda}$ be a family of constants such that for each $\lambda\in\Lambda$, $0\leq k_\lambda<1$ and $T:X\to \mathcal{K}(X)$ be an $\F$-contractive multi-function. Then, $T$ has a fixed point.
\end{corollary}
\begin{proof}
For each $x\in X$, $\lambda\in\Lambda$ and $\eta>0$, let $C_{\lambda,\eta}(x)= \{y\in Tx:d_\lambda(x,y)\leq d_\lambda(x,Tx)+\eta\}$. Since $\{d_\lambda\}_{\lambda\in\Lambda}$ satisfies condition (S), for each $x\in C$, $\lambda,\mu\in\Lambda$ and $\eta_1,\eta_2>0$, there exists $\nu\in\Lambda$ such that $C_{\nu,\eta}(x)\subseteq C_{\lambda,\eta_1}(x)\cap C_{\mu,\eta_2}(x)$, where $\eta=\min\{d_\lambda(x,Tx)+\eta_1,d_\mu(x,Tx)+\eta_2\}$. Thus, $\{C_{\lambda,\eta}(x)\}_{\lambda\in\Lambda}$ is a filter base of compact subsets of $Tx$ and hence $\bigcap_{\lambda\in\Lambda,\eta>0}C_{\lambda,\eta}(x)$ is nonempty. Consequently, for each $x\in X$, there exists $f(x)\in Tx$ such that $d_\lambda(x,Tx)=d_\lambda(x,f(x))$, for all $\lambda\in\Lambda$. We have
$$
d_\lambda(f(x),Tf(x))\leq H^\lambda(Tx,Tf(x))\leq k_\lambda d_\lambda(x,f(x))
$$
and hence, by defining $\varphi_\lambda:X\to\rr$ as $\varphi_\lambda(x)=d_\lambda(x,Tx)/(1-k_\lambda)$, we obtain
$$
d_\lambda(x,f(x))\leq \varphi_\lambda(x)-\varphi_\lambda(f(x)).
$$
Moreover $T$ is $\H$-continuous. Hence, Theorems \ref{t2} and \ref{t3} imply that $\varphi_\lambda$ is lower semi-continuous and from Corollary \ref{c12}, there exists $x^*\in X$ such that $f(x^*)=x^*$. Therefore, $x^*\in Tx^*$ and the proof is complete.
\end{proof}

\begin{remark}
Note that the functions $\varphi_\lambda$ in Corollary \ref{c13} depend on $\lambda$ not only through  $\nu_\lambda=1-k_\lambda$, but also due to $ d_\lambda $. Hence, this corollary cannot be obtained from Theorem 1 by Mizoguchi in \cite{Mi90}.
\end{remark}

Corollary \ref{c13} enables us to obtain the following old result by Tarafdar in \cite{Ta74}.

\begin{corollary}\label{c14} Let $\F=\{k_\lambda\}_{\lambda\in\Lambda}$ be a family of constants such that for each $\lambda\in\Lambda$, $0\leq
k_\lambda<1$ and $f:X\to X$ be an $\F$-contractive function, i.e.\ for any $x,y\in X$, $d_\lambda(f(x),f(y))\leq k_\lambda d_\lambda(x,y)$. Then, there
exists a unique $x^*\in X$ such that $f(x^*)=x^*$.
\end{corollary}
\begin{proof} From Corollary \ref{c13}, there exists $x^*\in X$ such that $f(x^*)=x^*$.
Suppose $x^{**}$ is another fixed point of $f$. Hence
$
d_\lambda(x^*,x^{**})=d_\lambda(f(x^*),f(x^{**}))\leq k_\lambda
d_\lambda(x^*,x^{**})
$
and hence uniqueness follows due to for all $\lambda\in\Lambda$, $d_\lambda(x^*,x^{**})=0$. Therefore, the proof is complete.
\end{proof}

Next we state an extended version of the Bishop-Phelps theorem as follows.

\begin{theorem}\label{t16} For each $\lambda\in\Lambda$, let $\varphi_\lambda:X\to\rr$ be a lower semi-continuous  and bounded below function. Then,  for
each $x_0\in X$, there exists $x^*\in X$ such that the following two conditions hold:
\begin{description}
  \item[(\ref{t16}.1)] for any $\lambda\in\Lambda$, $\varphi_\lambda(x^*)+d_\lambda(x_0,x^*)\leq\varphi_\lambda(x_0)$; and
  \item[(\ref{t16}.2)] for each $x\in X$ with $x\neq x^*$, there exists $\mu\in \Lambda$ satisfying $\varphi_\mu(x^*)<\varphi_\mu(x)+d_\mu(x,x^*)$.
\end{description}
\end{theorem}
\begin{proof}
Let $x_0\in X$. From Lemma \ref{le10}, there exists a maximal element $x^*\in X$ such that  (\ref{t16}.1) is satisfied. The maximality of $x^*$ implies condition (\ref{t16}.2), therefore the proof is complete.
\end{proof}

Also, a generalization of Ekeland's variational principle is given below. A function $\varphi:X\to\rr\cup\{\infty\}$ is said to be proper, whenever
there exists $x_0\in X$ such that $\varphi(x_0)<\infty$.

\begin{theorem}\label{t17} For each $\lambda\in\Lambda$, let $\varphi_\lambda:X\to\rr\cup\{\infty\}$ be a proper,  lower semi-continuous and bounded
below function. Then,  for each $x_0\in X$ and every family $\{\delta_\lambda\}_{\lambda\in\Lambda}$ of positive real numbers such that
$\varphi_\lambda(x_0)\leq\inf\varphi_\lambda(X)+\delta_\lambda$, there exists $x^*\in X$ satisfying the following three conditions:
\begin{description}
  \item[(\ref{t17}.1)] for any $\lambda\in\Lambda$, $\varphi_\lambda(x^*)\leq\varphi_\lambda(x_0)$;
  \item[(\ref{t17}.2)] for any $\lambda\in\Lambda$, $d_\lambda(x_0,x^*)\leq\delta_\lambda$; and
  \item[(\ref{t17}.3)] for each $x\in X$ with $x\neq x^*$, there exists $\mu\in \Lambda$ satisfying $\varphi_\mu(x^*)<\varphi_\mu(x)+d_\mu(x,x^*)$.
\end{description}
\end{theorem}
\begin{proof}
Let $x_0\in X$ and $\{\delta_\lambda\}_{\lambda\in\Lambda}$ be a family of positive  real numbers such that
$\varphi_\lambda(x_0)\leq\inf\varphi_\lambda(X)+\delta_\lambda$. Since each $\varphi_\lambda$ is proper, for each $\lambda\in\Lambda$,
$\inf\varphi_\lambda(X)<\infty$ and thus, $\varphi_\lambda(x_0)<\infty$. Let $Y=\bigcap_{\lambda\in\Lambda}\{x\in
X:\varphi_\lambda(x)\leq\varphi_\lambda(x_0)\}$ and $\widetilde{\varphi}_\lambda:Y\to\rr$ the restriction of $\varphi_\lambda$ to $Y$. Notice $Y$ is
closed in $X$ and consequently $Y$ is complete. Hence, we can apply Theorem \ref{t16} to the family
$\{\widetilde{\varphi}_\lambda\}_{\lambda\in\Lambda}$ and, since $x_0\in Y$, there exists $x^*\in Y$ such that for any $\lambda\in\Lambda$,
\begin{equation}\label{e1}
\widetilde{\varphi}_\lambda(x^*)+d_\lambda(x_0,x^*)\leq\widetilde{\varphi}_\lambda(x_0).
\end{equation}
Moreover, for each $x\in Y$ with $x\neq x^*$, there exists $\mu\in\Lambda$ such that
\begin{equation}\label{e2}
\widetilde{\varphi}_\mu(x^*)<\widetilde{\varphi}_\mu(x)+d_\mu(x,x^*).
\end{equation}
Since $x^*\in Y$, condition (\ref{t17}.1) holds. From (\ref{e1}) and the
assumption, for each $\lambda\in\Lambda$, we have $$ \widetilde{\varphi}_\lambda(x^*)+d_\lambda(x_0,x^*)\leq\widetilde{\varphi}_\lambda(x_0)\leq
\inf\varphi_\lambda(X)+\delta_\lambda \leq \varphi_\lambda(x^*)+\delta_\lambda, $$ which implies condition (\ref{t17}.2). Let us prove condition (\ref{t17}.3). From
(\ref{e2}) this condition  holds for all $x\in Y$. Let $x\in X\setminus Y$. Hence, there exists $\mu\in\Lambda$ such that
$\varphi_\mu(x_0)<\varphi_\mu(x)\leq \varphi_\mu(x)+d_\mu(x,x^*)$ and since $\varphi_\mu(x^*)\leq\varphi_\mu(x_0)$, we have $
\varphi_\mu(x^*)<\varphi_\mu(x)+d_\mu(x,x^*)$. Therefore, for each $x\in X$, $\varphi_\mu(x^*)<\varphi_\mu(x)+d_\mu(x,x^*)$, which  concludes the
proof.
\end{proof}

The following corollary is useful to obtain some fixed point results under some inwardness  conditions.

\begin{corollary}\label{c18}  Let $K$ be a closed subset of $X$, $f:K\to K$ an arbitrary function and $T:K\to\C(X)$ a multi-function.  Suppose  the following two conditions hold:
\begin{description}
\item[(\ref{c18}.1)] $T$ is weakly upper semi-continuous, and
\item[(\ref{c18}.2)] for each $\lambda\in\Lambda$ and $x\in X$, $d_\lambda(f(x),Tf(x))\leq d_\lambda(x,Tx) +r_\lambda d_\lambda(x,f(x))$, whenever $r_\lambda<0$.
\end{description}
Then, there exists $x^*\in X$ such that $f(x^*)=x^*$.
\end{corollary}
\begin{proof}
Since $K$ is a closed subset of $X$ and $X$ is complete, $K$ is
 complete. For each  $\lambda\in\Lambda$, let $\varphi_\lambda:K\to\rr$ such that $\varphi_\lambda(x)=-d_\lambda(x,Tx)/r_\lambda$. From
 (\ref{c18}.1), $\varphi_\lambda$ is lower semi-continuous and bounded below, and from (\ref{c18}.2) we have, $ d_\lambda(x,f(x))\leq
 \varphi_\lambda(x)-\varphi_\lambda(f(x))$. Therefore, from Corollary \ref{c12}, there exists $x^*\in X$ such that $f(x^*)=x^*$  and the proof is
 complete.
\end{proof}

\section{Applications to locally convex spaces}
In this section, $X$ denotes a real or complex locally convex space endowed with a family $\{\n{\cdot}_\lambda\}_{\lambda\in\Lambda}$ of semi norms defining the topology of $X$ and for each $Y\subseteq X$, $\Q(Y)$ stands for the family of all nonempty, compact, convex subsets of $X$. Here, for each $\lambda\in\Lambda$, $d_\lambda$ is defined as $d_\lambda(x,y)=\n{x-y}_\lambda$, for $(x,y)\in X\times X$.

\begin{lemma}\label{l20}
Let $A$ and $B$ be two subsets of $X$ and $\beta$ be a scalar. Then
\begin{itemize}
  \item [(i)] $\alpha(A+B)\leq\alpha(A)+\alpha(B)$.
  \item [(ii)] $\alpha(\beta A)=|\beta|\alpha(A)$.
  \item [(iii)] $\alpha(\mathrm{co}(A))=\alpha(A)$.
\end{itemize}
\end{lemma}
\begin{proof}
Conditions (i) and (ii) are easy to verify and we only prove condition (iii). We need to prove that $\alpha(co(A))\leq\alpha(A)$. Let $\epsilon=\alpha(A)$, $\lambda\in\Lambda$, $\eta>0$ and $x_1,\dots,x_r\in X$ such that $A\subseteq \mathrm{B}_\lambda(x_1,\epsilon+\eta)\cup\cdots\cup\mathrm{B}_\lambda(x_r,\epsilon+\eta)$. We have
$$
\mathrm{co}(\mathrm{B}_\lambda(x_1,\epsilon+\eta)\cup\mathrm{B}_\lambda(x_2,\epsilon+\eta))=\{tx_1+(1-t)x_2:0\leq t\leq1\}+\mathrm{B}_\lambda(0,\epsilon+\eta)
$$
and there exists $0=t_0<t_1\cdots<t_s=1$ a partition of $[0,1]$ such that $(t_i-t_{i-1})\n{x_1-x_2}_\lambda<\eta$. Hence
\begin{equation}\label{e3}
    \mathrm{co}(\mathrm{B}_\lambda(x_1,\epsilon+\eta)\cup\mathrm{B}_\lambda(x_2,\epsilon+\eta))\subseteq \bigcup_{i=1}^s\{t_ix_1+(1-t_i)x_2:0\leq t\leq1\}+\mathrm{B}_\lambda(0,\epsilon+2\eta)
\end{equation}
and since
$$
\mathrm{co}(\mathrm{B}_\lambda(x_1,\epsilon+\eta)\cup\cdots\cup\mathrm{B}_\lambda(x_r,\epsilon+\eta))\subseteq
\mathrm{co}\{\mathrm{B}_\lambda(x_1,\epsilon+\eta)\cup\mathrm{co}(\mathrm{B}_\lambda(x_2,\epsilon+\eta)\cup\cdots\cup\mathrm{B}_\lambda(x_r,\epsilon+\eta))\},
$$
it follows from induction and \eqref{e3} that there exist a finite subset $F$ of $X$ such that
$$
\mathrm{co}(\mathrm{B}_\lambda(x_1,\epsilon+\eta)\cup\cdots\cup\mathrm{B}_\lambda(x_r,\epsilon+\eta))\subseteq F+\mathrm{B}_\lambda(0,\epsilon+2\eta).
$$
Consequently, $\alpha(\mathrm{co}(A))\leq\alpha(F+\mathrm{B}_\lambda(0,\epsilon+2\eta))\leq \epsilon+2\eta$, which completes the proof due to $\eta>0$ is arbitrary.
\end{proof}
\begin{lemma}\label{l19}
Let $C\in\Q(X)$ and $T:C\to\Q(C)$ be a continuous multi-function. Then, $T$ has a fixed point.
\end{lemma}
\begin{proof}
It directly follows from Theorem 2 by Fan in \cite{Fa61}.
\end{proof}

The following result extends, to locally convex spaces and for continuous multi-functions, a known theorem by Darbo in \cite{Da55}.
\begin{theorem}
Let $D$ be a bounded, closed and convex subset of $X$ and $T:D\to \Q(D)$ be a condensing multi-function. Then, $T$ has a fixed point.
\end{theorem}
\begin{proof}
Let $x_0\in X$ and $\Sigma=\{K\in\mathcal{CB}(X):x_0\in K\subset D, K\mbox{ convex and } T(K)\subseteq K\}$. Due to $D\in \Sigma$, we have $\Sigma\neq\oo$. Let $B=\bigcap_{K\in\Sigma}K$ and $C=\overline{\mathrm{co}}(\{x_0\}\cup T(B))$. We have $T(B)\subseteq\bigcap_{K\in\Sigma}T(K)\subseteq B$ and $x_0\in B$. Moreover, since $B$ is closed and convex, $C\subseteq B$. Thus $T(C)\subseteq T(B)\subseteq  C$, $C\in\Sigma$ and $B=C$. From property c) and (iii) in Lemma \ref{l20}, we obtain $\alpha(C)=\alpha(T(B))=\alpha(T(C))$  and since $T$ is condensing, we have $\alpha(C)=0$. Since $C\in\Q(X)$ and $T(C)\in\Q(C)$, it follows from Lemma \ref{l19} that $T$ has a fixed point, which completes the proof.
\end{proof}

For each $K$ subset of $X$ and $x\in K$, the inner set of $K$ at $x$ is defined as
$$
I_K(x)=x+\{c(y-x):y\in K,c\geq1\}.
$$
Also we define  the envelope of $I_K(x)$ as
$$
\widetilde{I}_K(x)=\bigcup_{z\in I_K(x)}\bigcap_{\lambda\in\Lambda,\eta>0}\{t\in X:\n{t-z}_\lambda\leq\eta\n{t-x}_\lambda\}.
$$
Let $T:K\to \C(X)$ be a function. We say $T$ is inward (respectively, weakly inward), if for each $x\in K$, $T(x)\in I_K(x)$ (respectively, $T(x)\in \widetilde{I}_K(x)$).

Theorem \ref{tt} below is an extension to locally convex spaces of an result by Mart\'inez-Y\'a\~{n}ez in \cite{Ma91}. See also \cite{Re72}.

\begin{theorem}\label{tt}
Let $\F=\{k_\lambda;\lambda\in\Lambda\}$ be a family of nonnegative constants  and $T:K\subseteq X\to X$ be a weakly inward $\F$-contractive function. Then, $T$ has a fixed point.
\end{theorem}
\begin{proof}
For each $\lambda\in\Lambda$, choose $\epsilon_\lambda>0$ such that $k_\lambda<(1-\epsilon_\lambda)/(1+\epsilon_\lambda)$. Since $T$ is weakly inward, there exists $f:K\to K$ and $c:K\to\rr$ such that, for any $\lambda\in\Lambda$ and $x\in X$,
\begin{equation}\label{e4}
    \n{T(x)-x-c(x)(f(x)-x)}_\lambda\leq\epsilon_\lambda \n{T(x)-x}_\lambda
\end{equation}
with $c(x)\geq1$.  Let $h(x)=1/c(x)$, $w(x)=(1-h(x))x+h(x)T(x)$, fix $\lambda\in\Lambda$ and observe that
$$
\n{w(x)-x}_\lambda=h(x)\n{T(x)-x}_\lambda\quad\mbox{and}\quad\n{w(x)-f(x)}_\lambda\leq\epsilon_\lambda h(x)\n{T(x)-x}_\lambda.
$$
Hence
$
\n{x-f(x)}_\lambda\leq(1+\epsilon_\lambda)\n{w(x)-x}_\lambda
$
and accordingly
$$
\begin{array}{ccl}
  \n{f(x)-T(f(x))}_\lambda & \leq & \n{f(x)-w(x)}_\lambda+\n{w(x)-T(x)}_\lambda+\n{Tx-T(f(x))}_\lambda \\
   & \leq & \epsilon_\lambda h(x)\n{T(x)-x}_\lambda+(1-h(x))\n{T(x)-x}_\lambda +k_\lambda\n{x-f(x)} \\
   & = &  \n{T(x)-x}_\lambda +k_\lambda\n{x-f(x)}-(1-\epsilon_\lambda)\n{w(x)-x}_\lambda.\\
\end{array}
$$
Consequently,
\begin{equation}\label{e5}
    \n{f(x)-T(f(x))}_\lambda \leq  \n{T(x)-x}_\lambda+\left(k_\lambda -\frac{1-\epsilon_\lambda}{1+\epsilon_\lambda}\right)\n{x-f(x)}_\lambda,
\end{equation}
and thus, Corollary \ref{c18} implies that there exists $x^*\in X$ such that $f(x^*)=x^*$. Since $\epsilon_\lambda\leq1$, from \eqref{e4} we have, for any $\lambda\in\Lambda$, $\n{T(x^*)-x^*}_\lambda=0$. Therefore, $x^*$ is a fixed point of $T$ and the proof is complete.
\end{proof}

\paragraph{Acknowledgments} This work was partially supported by FONDECYT grant 1120879 from the Chilean government.

\end{document}